\documentclass[a4paper, 11pt]{amsart}
\usepackage{graphicx}
\usepackage{amssymb}
\usepackage[margin=3cm]{geometry}
\usepackage{color}
\usepackage{enumerate}
\usepackage[T1]{fontenc}
\usepackage[utf8]{inputenc}
\usepackage{mathtools}
\usepackage{comment}
\usepackage{float}
\usepackage{color}
\usepackage{mathrsfs}
\usepackage[italian, english]{babel}
\newtheorem{theorem}{Theorem}[section]
\newtheorem{lemma}{Lemma}[section]
\newtheorem{definition}{Definition}
\newtheorem{remark}{Remark}
\newtheorem{corollary}{Corollary}[section]
\numberwithin{equation}{section}

\def\Xint#1{\mathchoice
{\XXint\displaystyle\textstyle{#1}}%
{\XXint\textstyle\scriptstyle{#1}}%
{\XXint\scriptstyle\scriptscriptstyle{#1}}%
{\XXint\scriptscriptstyle\scriptscriptstyle{#1}}%
\!\int}
\def\XXint#1#2#3{{\setbox0=\hbox{$#1{#2#3}{\int}$ }
\vcenter{\hbox{$#2#3$ }}\kern-.6\wd0}}

\def\dashint{\Xint-}

\begin{document}

\title[]{Remarks on Sobolev-Morrey-Campanato spaces   defined  on $C^{0,\gamma}$ domains}

\author{Pier Domenico Lamberti}%
\address{Universit\`a degli Studi di Padova, Dipartimento di Matematica ``Tullio Levi-Civita'', Via Trieste 63, 35121 Padova, Italy}%
\email{lamberti@math.unipd.it}%
\author{Vincenzo Vespri}%
\address{Universit\`a degli Studi di Firenze, Dipartimento di Matematica e Informatica ``Ulisse Dini'', Viale Morgagni, 67/a,
50134 Firenze, Italy}%
\email{vespri@math.unifi.it}%


\maketitle

\noindent {\bf Abstract.}  We discuss a few old results  concerning embedding theorems for Campanato and Sobolev-Morrey spaces adapting the formulations
to the case of domains of class $C^{0,\gamma}$, and we present  more recent  results concerning  the extension of  functions from Sobolev-Morrey spaces  defined on those domains.  As a corollary of the extension theorem we obtain an embedding theorem for Sobolev-Morrey spaces on arbitrary  $C^{0,\gamma}$ domains. 

\section{ Introduction}

In the seminal papers \cite{campa1, campa2} Sergio Campanato introduced the spaces that nowadays are named after him,  and used them to prove embedding theorems for Sobolev-Morrey spaces defined on bounded open sets  $\Omega$ in $\mathbb{R}^n$. In particular, it was proved that if $f$ is a function belonging to the Campanato space $ {\mathcal L}^{\lambda}_{p }(\Omega )$ with $n<\lambda $  (and  $\lambda  \le n+p$)
then $f$ is H\"{o}lder continuous with exponent $(\lambda -n)/p  $ that is for some $c>0$
\begin{equation}
\label{intro1}
|f(x)-f(y)| \le c |x-y|^{\frac{\lambda -n}{p}},
\end{equation}
for all $x,y\in \Omega $, and it was also proved that if $f$ is a function in the Sobolev-Morrey space  $W^{l,\lambda}_p(\Omega )$  with $0\le \lambda <n$, $  n-\lambda <pl $  (and $pl < n-\lambda  +p$) then $f$ is H\"{o}lder  continuous with exponent $ l+\frac{\lambda -n}{p}$ that is   for some $c>0$
\begin{equation}
\label{intro2}
|f(x)-f(y)| \le c |x-y|^{l+\frac{\lambda -n}{p}},
\end{equation}
for all $x,y\in \Omega $. Here, for simplicty,  $l\in \mathbb{N}$ and  $W^{l,\lambda}_p(\Omega )$  is the space of functions with weak derivatives up to order $l$ in the classical Morrey space 
$L_p^{\lambda }(\Omega )$, but we note that  the focus of  \cite{campa1, campa2}  was mainly on the case of fractional order of smoothness $l$ since the case of integer exponents was already discussed in \cite{greco, nirenberg}.   See Section~\ref{embe} for precise definitions.  

The importance of these spaces is evident in  regularity theory and  harmonic analysis.
The classical  regularity approach was based on the singular integrals theory approach introduced by A.P.~Calder\'{o}n and A.~Zygmund~\cite{CZ}. Using this approach  based on the heat kernel, J.~Nash~\cite{N} was able to solve the XIX Hilbert problem about the analyticity of the solutions to regular problems in the calculus of variations. One year before, E.~De Giorgi~\cite{DG} proved the same result with a different approach. He introduced  a suitable function space, the so-called De Giorgi class, proved that any solution to regular problems in the calculus of variations belongs to this class  and showed the embeddings of the De Giorgi classes in the space of H\"older continuous  functions. It was a natural question to ask if this approach could recover the classical  Calder\'{o}n and Zygmund theory  for equations with regular coefficients (i.e.,  continuous or H\"older continuous coefficients). This question was proposed by Ennio De Giorgi and  Guido Stampacchia and solved by S. Campanato with the introduction of the Campanato spaces.  The regularity in $L^p$ spaces was proved by S.~Campanato and G.~Stampacchia~\cite{CS}  with the supplementary hypotheses of the H\"older continuity of the coefficients (for a proof of such a result, with only the assumption of the continuity of the coefficients, see \cite{CTV}). These spaces were used for proving regularity of solutions to elliptic/parabolic  systems/equations in variational/nonvariational form  of the second (and higher) order (see, for instance \cite{C} and \cite{G}).

The other important field of application of these function spaces is harmonic analysis. T.~Walsh~\cite{W} proved that the dual space of the Hardy space $H^p (R^ N)$ is exactly the Campanato space. The theory of Hardy spaces $H^p (R^N)$  has  important applications in harmonic analysis and partial differential equations (for instance,  see \cite{FS, S, S1, SW}). We recall that when $p \in (1, \infty)$, $L^ p (R^ N$) and $H^p (R^N)$  are isomorphic; but  when $p \in (0, 1]$, some of singular integrals (for example, Riesz transforms) are bounded on $H^p (R^N)$  but not on $L^p (R^N)$  and this fact makes the space $H^p (R^N)$  the right space where  to study  the theory of the boundedness of operators.   In \cite{FS}  C.~Fefferman and E.M.~Stein characterised   the Hardy space $ H^1 (R^N)$  as the predual of the space BMO$(R^N)$. The atomic and the molecular characterizations of  $H^p (R^N)$   and their applications were studied by many authors; see, for example, \cite{Co, Co1,  L, NS, SW, TW}. These  characterizations (atomic and  molecular) are necessary to extend   the theory of Hardy spaces  to spaces of homogeneous type in the sense of R.R.~Coifman and  G.~Weiss \cite{CoW, CoW1}, which is, by  far, one of the   most  general setting for singular integrals.

Going back to the initial work of S. Campanato, we note that inequality \eqref{intro1} was proved under the assumption that $\Omega$ satisfies the so-called property $(A)$ which requires the existence of a constant $M>0$ such that 
\begin{equation}
\label{conditiona}
|\Omega  \cap B(x,r)|\geq Mr^n\, ,
\end{equation}
for all $x\in \Omega$ and all $r>0$ smaller than the diameter of $\Omega$. 

 Inequality \eqref{intro2} was obtained under the stronger assumption that $\Omega$ is of class $C^{0,1}$, which means that, locally at the boundary,  $\Omega$ can be represented as the subgraph of a Lipschitz continuous function (possibly after a rotation of coordinates).  
 
 Note that for $\lambda =0$ we have $W^{l,\lambda}_p(\Omega )= W^{l}_p(\Omega )$ and inequality \eqref{intro2} is the celebrated Sobolev-Morrey inequality. 
 
 In this paper, we consider the case of open sets $\Omega$ of class $C^{0,\gamma}$ with $0<\gamma \le 1$ which means that the functions describing the boundary of $\Omega$ are H\"{o}lder continuous of exponent $\gamma$. It is a matter of folklore that passing from Lipschitz to H\"{o}lder continuity assumptions at the boundary of an open set is highly nontrivial (see e.g., the recent paper \cite{lapi}), and it is interesting to note that also S. Campanato himself devoted his paper \cite{campa3} to the study of embeddings for Sobolev spaces  on open sets  with power-type cusps at the boundary.    
 We refer to the extensive monograph \cite{mazpob} for a recent introduction to the analysis of function spaces on irregular domains.  We also refer to the classical monograph \cite{kufner} for an introduction to Morrey-Campanato spaces on regular domains. 
 
  Broadly speaking,  one may say that classical embedding theorems for Sobolev-Morrey-Campanato spaces hold on $C^{0,\gamma}$ domains provided 
 one replaces (in the inequalities  involved)  the dimension $n$  of the underlying space by  $n_{\gamma}=(n-1)/\gamma +1$, a fact which also appeared in  \cite{campa3}. 
 It is important to note that $n_{\gamma}>n$ if $\gamma <1$, and this typically leads to a deterioration in the estimates. 
 For instance,  if $\Omega$ is a domain with  outer  power-type   cusps with exponent $\gamma$, property (A) above holds provided  $n$ is replaced by $n_{\gamma}$ in the right-hand side of \eqref{conditiona}.  On the other hand, we observe that if one wishes to control $|\Omega  \cap B(x,r) |$ from above, the best one can do is to write 
 $|\Omega  \cap B(x,r) |\le cr^n   $, since it is impossible here to use $n_{\gamma}$.  This discrepancy between the  upper and lower  bounds for $|\Omega  \cap B(x,r) |$, indicates that the standard Euclidean metric is not suitable to deal with cusps and suggests to adapt the balls $B(x,r)$ to the type of domain under consideration.  For example, if $\Omega$ is given  by the cusp 
 \begin{equation}
 \label{cusp}
 \{(\bar x, x_n)\in \mathbb R^n:\ \bar x \in \mathbb{R}^{n-1}, \ x_n>|\bar x|^{\gamma}  \}
 \end{equation} with $\gamma <1$, then one should replace the Euclidean ball $B(x,r)$ by the 
anisotropic  ball 
\begin{equation}
\label{ballintro}B_{\gamma}(x,r)=\bigl\{y\in R^n:\ |\bar x-\bar y|<r^{\frac{1}{\gamma}},\ |x_n-y_n|<r     \bigr\}  
\end{equation}
in which case $ |\Omega  \cap B_{\gamma }(x,r) |$ is asymptotic to $r^{n_{\gamma}}   $ as $r\to 0$, and the discrepancy above disappears.  Accordingly, in the right-hand side of inequalities \eqref{intro1},  \eqref{intro2} one has to replace the Euclidean distance  $|x-y|$ by the anisotropic one $|\bar x -\bar y|^{\gamma} +|x_n-y_n| $.  This idea was already used  by G.C.~Barozzi~\cite{ba} where some results of S. Campanato are extended to the case of domains with power-type cusps, and was further extended  
by Giuseppe Da Prato in the fundamental paper  \cite{da} where more general metrics were considered. We also note that final results for domains satisfying horn-type conditions are  contained in the classical monograph \cite{be1, be2}.

 Although the existing literature  seems to provide a complete picture of this subject, we have found it quite surprising that some results contained in the above mentioned papers, incorporate quite restrictive assumptions.  In particular, in the analysis of inequality \eqref{intro2} for 
anisotropic metrics,  \cite[Theorem~3]{ba}  eventually  assumes for simplicity that $\Omega$ is a parallelepiped, \cite[Theorem~4.1]{da} assumes that $\Omega$ is convex
 and the estimate in \cite[Theorem~27.4.2]{be2} is  proved for all $x,y\in \Omega$ such that the segment $[x,y]$ is contained in $\Omega$.  
 
 A different approach to the analysis of function spaces in domains of class $C^{0,\gamma}$ was suggested by Victor I. Burenkov in \cite{burpaper2, burpaper1} where he defined a new extension operator 
 which, contrary to other  classical extension operators, allows to deal not only with Lipschitz domains but also  with  $C^{0,\gamma}$ domains (as well as with anisotropic Sobolev spaces and extensions from manifolds of dimension $m<n$).  Note that the flexibility of Burenkov's Extension Operator has been recently exploited in \cite{fala} where it is proved that this operator preserves general Sobolev-Morrey spaces, including the case of the classical Sobolev-Morrey spaces   $W^{l,\lambda}_p(\Omega )$.  
 
 If $\gamma <1$ then deterioration in the smoothness of the extended functions is  expected and, in fact, Burenkov's Extension Operator maps 
 the Sobolev space $W^{l}_p(\Omega)$ to the Sobolev space $W^{[\gamma l]}_p(\mathbb{R}^n)$ where $[\gamma l]$ is the integer part of $\gamma l$.  The exponent $[\gamma l]$ is sharp (in terms of Sobolev spaces). Thus, having a function extended to the whole of $\mathbb{R}^n$ allows to apply embedding theorems  in $\mathbb{R}^n$  and eventually to return to $\Omega$ by mere restriction. 
 Although the target space $W^{[\gamma l]}_p(\mathbb{R}^n)$ is sharp, it is observed  already in  \cite{burpaper1} that in general   the embedding theorems proved via this procedure are not sharp since the deterioration given by $[\gamma l]$ is too much for this purpose. However, this procedure has the advantage of giving at least  some information even in most difficult cases. 

 The goal of the present paper is twofold. First, we revise the above mentioned  old results by adapting their formulation to the case of elementary domains of class $C^{0,\gamma}$. In passing,  we also indicate how it is  possible to replace the convexity assumption  in \cite[Theorem~4.1]{da} by the assumption that  a Poincar\'{e} inequality for balls holds, see Theorem~\ref{dapratovariant}. Secondly, we indicate how to adapt  the proofs of \cite{fala} to  the case of domains of class $C^{0,\gamma}$, in order to  prove that Burenkov's Extension Operator maps the Sobolev-Morrey space  $W^{l,\lambda}_p(\Omega )$ to the Sobolev-Morrey space  $W^{[\gamma l],\gamma \lambda}_p(\mathbb R^n )$, analysing also the case of Morrey norms defined by even more general weights, see Theorem~\ref{sobolevmorreybis}.   Note the extra deterioration in the Morrey exponent which passes from $\lambda $ to $\gamma \lambda$. Moreover, we apply this extension result to recover an estimate of  type \eqref{intro2} in domains of class $C^{0,\gamma}$, see Corollary~\ref{maincor}. We observe  that, although the new estimate is not  sharp, it is obtained without any extra geometric assumptions on $\Omega$ or on the points $x,y\in \Omega$ as done by other authors.  
 
 It is important to observe that our extension result is obtained by using Morrey norms involving Euclidean balls both in $\Omega$ and in $\mathbb{R}^n$, even though
 for elementary domains of form \eqref{cusp} it would be natural to use anisotropic balls of  type \eqref{ballintro} in $\Omega$. This is due to technical reasons involved in our proofs,  which prevents us from controlling reflected balls
 in an anistropic way, see Lemma~\ref{inc0}. On the other hand, since our final goal is to deal with general open sets $\Omega$ of class $C^{0,\gamma}$ (where cusps may have a different orientation depending on the part of the boundary under consideration), in principle there is no special reason why one should use the balls of  type \eqref{ballintro} in the whole of $\Omega$.  Thus, either one changes the definition of the Morrey spaces, adapting balls to the orientation of each local chart or uses, for uniformity, Euclidean balls in the whole of $\Omega$. Our approach eventually leads us to choose the second option. 
 
 With reference to the problem of the extension of Sobolev-Morrey spaces, besides \cite{fala}, we would also like to quote the papers \cite{koskzhou}, \cite{lavio}, \cite{vitolo}.

\section{Embedding theorems on elementary $C^{0,\gamma}$ domains   }
\label{embe}

In this paper the elements of ${\mathbb R}^n$, $n\geq 2$, are denoted by  $x=(\overline x,x_n)$ with $\overline x=(x_1,\dots , x_{n-1})\in{\mathbb R}^{n-1}$ and $x_n\in {\mathbb R}$. 
For  any   $\gamma\in ]0,1]$,  we consider the metric  $\delta_{\gamma}$ in ${\mathbb R}^n$ defined by
$$
\delta_{\gamma}(x,y)= \max\{ |\bar x-\bar y|^{\gamma}, |x_n-y_n| \}\, ,
$$ 
for all $x,y\in {\mathbb R}^N$ and we denote by $B_{\gamma}(x,r)$ the corresponding open balls of centre $x$ and radius $r$, that is
\begin{eqnarray*}\lefteqn{
B_{\gamma}(x,r)=\{y\in {\mathbb R}^n:\ \delta_{\gamma}(x,y)<r   \} } \\
& & \qquad\qquad\qquad=\bigl\{y\in R^n:\ |\bar x-\bar y|<r^{\frac{1}{\gamma}},\ |x_n-y_n|<r     \bigr\}   . 
\end{eqnarray*}
Note  that the Lebesgue measure of $B_{\gamma}(x,y)$ is given by  
$$
|B_{\gamma}(x,r)   | =2\omega_{n-1}r^{n_{\gamma}}    
$$
where
$$
n_{\gamma}= \frac{n-1}{\gamma}+1\, ,
$$
and $\omega_{n-1}$ is the measure of the unit ball in ${\mathbb R}^{n-1}$.
Note also  that $n_{\gamma} =  n+(n-1)(\frac{1}{\gamma}-1   )$, hence $n_{\gamma}\geq n$ and equality occurs if and only if either $n=1$ or $\gamma=1$.

Given  $p\in [1,\infty [$,  a  function $\phi : ]0,\infty [\to ]0,\infty [$  and an open set $\Omega$ in ${\mathbb R}^n$,  for all $f\in L^p (\Omega )$ we set  
\begin{equation*}
\| f \|_{L^{\phi }_{p, \gamma }(\Omega ) }:=\sup_{x\in \Omega}\sup_{r>0 }\left(\frac{1}{\phi(r)}\int_{B_{\gamma}(x, r)\cap \Omega}|f(y) |^pdy \right)^{\frac{1}{p}}
\end{equation*}
and 
\begin{equation*}
| f |_{{\mathcal L}^{\phi  }_{p, \gamma }(\Omega ) }:=\sup_{x\in \Omega}\sup_{r>0 }\left(\frac{1}{\phi(r)}\int_{B_{\gamma}(x, r)\cap \Omega}|f(y)    
-\dashint_{B_{\gamma}(x, r)\cap \Omega}   f(z)dz  |^pdy \right)^{\frac{1}{p}}\, .
\end{equation*}

The generalised Morrey spaces are defined by 
$$
L^{\phi }_{p, \gamma }(\Omega )=\{ f\in  L^p (\Omega ):\ \| f \|_{L^{\phi }_{p, \gamma }(\Omega ) }< \infty   \} ,
$$
and the generalised Campanato spaces are defined by 
$$
{\mathcal L}^{\phi }_{p, \gamma }(\Omega )=\{ f\in  L^p (\Omega ):\ | f |_{{\mathcal L}^{\phi }_{p, \gamma }(\Omega ) }< \infty   \} . 
$$

For any $l\in {\mathbb N}$, we  consider  also the Sobolev-Morrey spaces 
$$
W^{l,\phi }_{p,\gamma}(\Omega )=\{ f\in L^p (\Omega):\ D^{\alpha }f\in    L^{\phi }_{p, \gamma }(\Omega ),\ \forall |\alpha|\le l  \}
$$ 
endowed with the norm
$$
\|  f \|_{W^{l,\phi}_{p,\gamma}(\Omega )}=\sum_{|\alpha | \le l}\| D^{\alpha }f \|_{ L^{\phi }_{p, \gamma }(\Omega )}\, .
$$

If    $\lambda \geq 0$  and   $\phi (r)=\min \{r^{\lambda }, 1 \}$ for all  $r>0$ then the corresponding spaces  will be denoted by 
$L^{\lambda }_{p, \gamma }(\Omega )$, ${\mathcal L}^{\lambda }_{p, \gamma }(\Omega )$,  $W^{l,\lambda }_{p,\gamma}(\Omega )$.
Since $| \cdot |_{{\mathcal L}^{\lambda }_{p, \gamma }(\Omega ) } $ is a semi-norm, it is customary to endow  the  Campanato space  ${\mathcal L}^{\lambda }_{p, \gamma }(\Omega )$ with the norm defined by 
$$
\| f \|_{{\mathcal L}^{\lambda  }_{p, \gamma }(\Omega ) } := \|f \|_{L^p(\Omega)}+ | f |_{{\mathcal L}^{\lambda  }_{p, \gamma }(\Omega ) } \, ,
$$
for all $f\in {\mathcal L}^{\lambda  }_{p, \gamma }(\Omega )$.

 Note that ${L^{\lambda }_{p, 1 }(\Omega ) }$,  ${{\mathcal{L}}^{\lambda }_{p, 1 }(\Omega ) }$ are 
the classical Morrey and Campanato spaces respectively   (recall that ${L^{\lambda }_{p, 1 }(\Omega ) }$ contains only the zero function for $\lambda >n$ and it coincides with $L^{\infty}(\Omega)$ for $\lambda =n$ by the Lebesgue differentiation theorem, see  \cite{kufner} for more details concerning the limiting cases). \\

 We consider  elementary H\"older continuous   domains   $\Omega$  in ${\mathbb R}^n$ with exponent $\gamma\in ]0,1]$ of the form 
  \begin{equation}\label{eldom}\Omega = \{x=(\overline x, x_n)\in {\mathbb R}^n:\ \bar x\in W,\  a< x_n<\varphi(\overline x) \}\,,
\end{equation}
where $-\infty \le a<\infty $,  $W$ is a  smooth or convex open set in ${\mathbb R}^{n-1}$, and $\varphi :  W\to {\mathbb R}$ is  a H\"older continuous function with exponent $\gamma$ satisfying  the condition $\varphi (\bar x)> a+\delta$ for some $\delta >0$. 
In particular, 
there exists a positive constant $M$ such that
\begin{equation}\label{lip1}
|\varphi(\overline x)-\varphi(\overline y)| \le M|\overline x-\overline y|^{\gamma}\,, \,\,\forall\  \overline x, \overline y\in {\mathbb R}^{n-1}\,.
\end{equation}
The best constant $M$ in inequality (\ref{lip1}) is denoted by ${\rm Lip}_{\gamma} \varphi$. 
For $\gamma=1$ we obtain Lipschitz continuous domains.
 It is well known that Lipschitz continuous domains  satisfy the usual cone condition. 
Similarly,  H\"older continuous domains satisfy a generalisation of that condition which we call the cusp condition. Namely, for any 
$x\in {\mathbb R}^n$ and $h>0$, we set 
\begin{equation} 
C_{\gamma}(x, h, M)=\{y\in {\mathbb R}^N:\   x_n-h<  y_n<x_n-M|\bar y -\bar x|^{\gamma}  \}
\end{equation}
and we call it a cusp with exponent $\gamma$, vertex $x$, height $h$  and opening $M$.  
Then we can prove the following simple lemma which, by the way, is essential in order to apply the general results of \cite{ba, be2, da}. 

\begin{lemma} \label{lemmacusp}
Let $\gamma\in ]0,1]$ and  $\Omega $ be an elementary H\"older continuous   domain   in ${\mathbb R}^n$ as in \eqref{eldom} with $W={\mathbb R}^{n-1}$ and $a=-\infty$. Then 
for all $x\in \bar \Omega$ and $h>0$,   we have
\begin{equation}
 \label{lemmacusp0}C_{\gamma}(x, h, {\rm Lip}_{\gamma}\varphi )\subset  \Omega .
 \end{equation}
 Moreover, there exists $c>0$ depending
only on $n,\gamma $ and ${\rm Lip }_{\gamma}\varphi $ such that 
\begin{equation}
 \label{lemmacusp1}
|B_{\gamma}(x,r)\cap \Omega| \geq cr^{n_{\gamma}},
\end{equation}
for all $x\in \bar \Omega$ and $r>0$.
\end{lemma} 

\begin{proof}  Given a cusp $C_{\gamma}(x, h, {\rm Lip}_{\gamma}\varphi )$ as in the statement, for any point $y\in C_{\gamma}(x, h, {\rm Lip}_{\gamma}\varphi ) $
we have
$$
y_n<x_n- {\rm Lip}_{\gamma } \varphi\,  |\bar x-\bar y |^{\gamma}\le \varphi (\bar x)-     {\rm Lip}_{\gamma } \varphi\,  |\bar x-\bar y |^{\gamma}   \le \varphi (\bar y)\, ,
$$
where the third inequality follows from the H\"{o}lder continuity of $\varphi$. Thus, 
$C_{\gamma}(x, h, {\rm Lip}_{\gamma}\varphi )\subset \Omega$. 
Inequality \eqref{lemmacusp1}, easily follows from  \eqref{lemmacusp0}, the inclusion 
$C_{\gamma}(x, r, 1)\subset B_{\gamma}(x,r)$ and 
the fact that  $|C_{\gamma}(x, h, M)|=ch^{n_{\gamma}} $ where $c$ is a positive constant depending only on $n, \gamma ,M $.
\end{proof}

Given two function spaces $X(\Omega)$ $Y(\Omega)$, we write $X(\Omega)\simeq Y(\Omega)$ to indicate that any function $f\in X(\Omega )$
equals almost everywhere in $\Omega $ a function $g\in Y(\Omega) $ and viceversa, and that the two norms $\| \cdot \|_{X(\Omega )} $,
 $\| \cdot \|_{Y(\Omega )} $ are equivalent. Note that, for the sake of simplicity, two functions $f,g$ as above will be denoted by the same symbol (being aware of this identification is particularly important when stating H\"{o}lder continuity estimates). 

The following theorem can be deduced by the general result \cite[Theorem~3.1]{da} combined with inequality  \eqref{lemmacusp1} which guarantees that  $\Omega$ is of  type (A) as required  in  \cite[Theorem~3.1]{da}.   
Here,   $C^{0,\alpha }(\bar \Omega , \delta_{\gamma})$ denotes the space of H\"{o}lder continuous functions with exponent $\alpha$  with respect to the metric $\delta_{\gamma}$.

\begin{theorem}[Campanato-Da Prato]\label{campa}
Let $\Omega $ be a bounded elementary H\"{o}l\-der continuous domain with exponent $\gamma \in ]0,1]$, $\lambda >0$. The following 
statements hold:
\begin{itemize}
\item[(i)] If $\lambda < n_{\gamma}$ then  $ {\mathcal{L}}^{\lambda }_{p, \gamma }(\Omega )\simeq {L}^{\lambda }_{p, \gamma }(\Omega )$. 
\item[(ii)]  If $\lambda > n_{\gamma}$ then  $ {\mathcal{L}}^{\lambda }_{p, \gamma }( \Omega )\simeq C^{0,\alpha }(\bar \Omega , \delta_{\gamma})$ 
where 
$$
\alpha  =\frac{\lambda -n_{\gamma }}{p}\, ;
$$
in particular, there exists $c>0$ such that  for all   $f\in    {\mathcal{L}}^{\lambda }_{p, \gamma }( \Omega )$ and   for all $x,y\in \Omega $
we have 
\begin{equation}
\label{campa2}
|f(x)-f(y)|\le c  | f|_{ {\mathcal{L}}^{\lambda }_{p, \gamma }( \Omega )}  (|\bar x-\bar y|^{\gamma}+|x_n-y_n|)^{\frac{\lambda -n_{\gamma } }{p} }\, .
\end{equation}

\end{itemize}

\end{theorem}

The following result is direct application of a general result  in \cite[Theorem~27.4.2, Remark~27.4.3]{be2} combined  with inclusion \eqref{lemmacusp0} which guarantees that $\Omega$ satisfies the $\gamma$-horn condition described in \cite[p. 153]{be1}.    As customary, we denote by $[x,y]$ the segment connecting two points $x$ and $y$ in ${\mathbb R}^n$.

\begin{theorem}[Sobolev-Morrey Embedding for elementary $C^{0,\gamma }$ domains]  \label{besov}
 Let $\Omega $ be an elementary H\"{o}lder continuous domain with exponent $\gamma \in ]0,1]$. Let $l\in {\mathbb N}$, $\lambda  >0$, $p\in [1, \infty [$ be such that 
 $$
 pl> n_{\gamma} -\lambda 
 $$
 and\footnote{If viceversa $\gamma (l+\frac{\lambda -n_{\gamma } }{p}   )>1$  then one has Lipschitz continuity; in the case  $\gamma (l+\frac{\lambda -n_{\gamma } }{p}   )=1$  one gets H\"{o}lder continuity  with any exponent less than $1$.  } $\gamma (l+\frac{\lambda -n_{\gamma } }{p}   )<1$. 
 Then there exists $c>0$ such that for all $f\in W^{l,\lambda}_{p,\gamma}(\Omega ) $ and for all $x,y\in \Omega $ such that 
 $[x,y]\subset \Omega $ we have   
 \begin{equation}\label{besov1}
 |f(x)-f(y)|\le c  \| f\|_{ W^{l,\lambda}_{p,\gamma}(\Omega )   }  |x-y|^{\gamma \left(l+\frac{\lambda -n_{\gamma } }{p}   \right)}\, .
 \end{equation}
\end{theorem}

Note that by setting formally $l=0$ in  \eqref{besov1},  one essentially obtains  estimate \eqref{campa2}. 
It is interesting to observe that  the previous result (with minor modifications) was proved in \cite{ba} in the case of a parallelepiped.

\begin{theorem}[Barozzi]\label{barozzi}  
Let $\Omega $ be a parallelepiped  in ${\mathbb R}^n$ of the form 
$\Omega =\Pi_{i=1}^n]a_i,b_i[$ with $-\infty <a_i<b_i<\infty $ for all $i=1,\dots ,n$. Let 
 $\gamma \in ]0, 1]$, $l\in {\mathbb N}$, $\lambda  >0$, $p\in [1, \infty [$ be such that 
 $$
 pl> n_{\gamma} -\lambda \, 
 $$
 and such that $ l+\frac{\lambda -n_{\gamma } }{p}\le 1$. 
 Then for any $\epsilon>0$ there exists $c>0$ such that for all $f\in W^{l,\lambda}_{p,\gamma}(\Omega ) $ and for all $x,y\in \Omega $ we have
$$
 |f(x)-f(y)|\le c \|f \|_{W^{l,\lambda}_{p,\gamma}(\Omega ) }(|\bar x-\bar y|^{\gamma}+|x_n-y_n|)^{l+\frac{\lambda -n_{\gamma } }{p} -\epsilon}\, .
 $$
\end{theorem}

Moreover, the following theorem can be deduced by a more general result  obtained  by G. Da Prato in \cite[Theorem 4.1]{da} for $l=1$ in the case of a convex set.  

\begin{theorem}\label{daprato} Let $\Omega$ be a bounded convex domain  in ${\mathbb R}^n$.  Let 
 $\gamma \in ]0, 1]$, and $\eta =\frac{n_{\gamma}}{n}+n-n_{\gamma}    $.  Let  $\lambda  >0$, $p\in [1, \infty [$ be such that 
 $$
 p \eta > n_{\gamma} -\lambda \, .
 $$
 Then  there exists $c>0$ such that for all $f\in W^{1,\lambda}_{p,\gamma}(\Omega ) $ and for all $x,y\in \Omega $ we have
$$
 |f(x)-f(y)|\le c\| f\|_{W^{1,\lambda}_{p,\gamma}(\Omega ) } (|\bar x-\bar y|^{\gamma}+|x_n-y_n|)^{  \eta   +\frac{\lambda -n_{\gamma } }{p} }\, .
 $$
\end{theorem}

\begin{remark} We note that the constant $\eta $ in Theorem~\ref{daprato} replaces the constant $l=1$ in the previous theorems. Since 
$\eta <1$ for $\gamma <1$, we have a deterioration in the estimates. This seems to be due to the fact that the result in \cite[Theorem 4.1]{da}  is quite general and is stated in order to 
embrace more general types of metrics.    
\end{remark}

We now explain  where the exponent $\eta $ in Theorem~\ref{daprato} comes from.
The main ingredient is a quantitative Poincar\'{e}-Wirtinger inequality for bounded convex domains $B$ in ${\mathbb R}^n$, namely the inequality
\begin{equation}\label{poincare0}
\| f-f_B   \|_{L^p(B)}\le \left( \frac{\omega_n}{|B|}  \right )^{1-\frac{1}{n}}d^n\| \nabla f \|_{L^p(B)}, \ \ \forall f\in W^{1}_p(B),
\end{equation} 
where $\omega _n$ denotes the Lebesgue measure of the unit ball in ${\mathbb R}^n$,  $d$ denotes the Euclidean diameter of $B$ and 
$f_B=\dashint_Bf(x)dx$ (see, e.g., \cite[p.164]{gitr}).

 It follows from \eqref{poincare0} and  H\"{o}lder's inequality that if $\Omega $ is a convex domain in ${\mathbb R}^n$ and $ f\in W^{1}_p(\Omega )$ then 
 for all $x\in \Omega $ and $r>0$ we have  
\begin{eqnarray}\label{poincare1}\lefteqn{
\| f-f_{\Omega \cap B_{\gamma}(x,r)}   \|_{L^1(\Omega \cap B_{\gamma}(x,r))}} \nonumber \\
& & 
\qquad\qquad\qquad\le  \omega_n ^{1-\frac{1}{n}} |\Omega \cap B_{\gamma}(x,r) |^{\frac{1}{n}-\frac{1}{p}} d^n_r\| \nabla f \|_{L^p(\Omega \cap B_{\gamma}(x,r))},
\end{eqnarray} 
where $d_r$ denotes the Euclidean diameter of $\Omega \cap B_{\gamma}(x,r)$. 
If in addition we have that $\nabla f\in L^{\lambda }_{p,\gamma}(\Omega )$, we obtain
\begin{equation*}
\| f-f_{\Omega \cap B_{\gamma}(x,r)}   \|_{L^1(\Omega \cap B_{\gamma}(x,r))}\le  c |\Omega \cap B_{\gamma}(x,r) |^{\frac{1}{n}-\frac{1}{p}} d^n_r r^{\frac{\lambda}{p}}  
\end{equation*} 
hence 
\begin{equation}\label{poincare3}
\| f-f_{\Omega \cap B_{\gamma}(x,r)}   \|_{L^1(\Omega \cap B_{\gamma}(x,r))}\le  c r^{n_{\gamma} ( \frac{1}{n}-\frac{1}{p})}  r^{\frac{\lambda}{p}}  d^n_r
\end{equation} 
since the measure of $ |\Omega \cap B_{\gamma}(x,r) |$  is controlled from above by a multiple of $r^{n_{\gamma}}$. 

In the general framework of \cite{da}, it is then assumed that $d_r\le c r^{\beta }$ for some constant $\beta \geq 1$ which in our case is $\beta =1$
and cannot be better.  Keeping track of $\beta$, we obtain from \eqref{poincare3} that 
\begin{equation}
\| f-f_{\Omega \cap B_{\gamma}(x,r)}   \|_{L^1(\Omega \cap B_{\gamma}(x,r))}\le  c r^{n_{\gamma} ( \frac{1}{n}-\frac{1}{p})}  r^{\frac{\lambda}{p}}  r^{n\beta }
\end{equation} 
which means that $f\in {\mathcal{L}}^{\theta}_{1, \gamma }(\Omega ) $ with 
$$
\theta = \frac{n_{\gamma}}{n}+n\beta +  \frac{ \lambda -n_{\gamma }}{p}\, .
$$

If $\theta >n_{\gamma }$, that is  $ (n_{\gamma}/n+n\beta -n_{\gamma} )p>n_{\gamma}-\lambda    $, we deduce by Theorem~\ref{campa}~(ii) that 
$u\in C^{0,\alpha }(\bar \Omega , \delta_{\gamma} )$ with 
$$
\alpha =  \frac{n_{\gamma}}{n}+n\beta +\frac{ \lambda -n_{\gamma }}{p}  -n_{\gamma},
$$
which for $\beta =1  $ yields 
$$
\alpha = \eta +\frac{\lambda - n_{\gamma }}{p}. 
$$
This explains the appearance of $\eta $ in Theorem~\ref{daprato}. \\

We  now reformulate the statement of  \cite[Theorem 4.1]{da} in order to relax a bit the convexity assumption
on $\Omega$. Namely, 
assume that $\Omega$ is a bounded domain in ${\mathbb R}^n$    such that condition \eqref{lemmacusp1} is satisfied and such that the following $p$-Poincar\'{e}
inequality holds 
\begin{equation}\label{poin}
\dashint_{ \Omega \cap B_{\gamma}(x,r)} | f- f_{\Omega \cap B_{\gamma}(x,r)} |dx\le c_pr^{  \tilde \eta } \left(\dashint_{ \Omega \cap B_{\gamma}(x,\tau r) }|\nabla f|^pdx\right)^{\frac{1}{p}}, 
\end{equation}
 for all $ f\in W^1_p(\Omega)$ and $r>0$, where $\tau \geq 1$ and $   \tilde\eta >0$ are a fixed constants.
 In particular,  if  $f\in W^{1,\lambda}_{p,\gamma}(\Omega ) $ we have
\begin{eqnarray*}\lefteqn{
\| f-f_{\Omega \cap B_{\gamma}(x,r)}   \|_{L^1(\Omega \cap B_{\gamma}(x,r))}} \nonumber \\
& & 
\qquad\qquad\qquad\le c_pr^{  \tilde \eta  } |\Omega \cap B_{\gamma}(x,\tau r) |^{1-\frac{1}{p}} \| \nabla f \|_{L^p(\Omega \cap B_{\gamma}(x,r))}
 \nonumber \\
& & 
\qquad\qquad\qquad\le c r^{     \tilde \eta  } (\tau r)^{n_{\gamma}(1-\frac{1}{p})}  \| \nabla f \|_{L^p(\Omega \cap B_{\gamma}(x,r))}  \le c  r^{n_{\gamma}(1-\frac{1}{p}) +\frac{\lambda }{p}+    \tilde \eta  },  
\end{eqnarray*} 
for some $c>0$ independent of $r$.

This implies 
that $f\in {\mathcal{L}}^{\theta}_{1, \gamma }(\Omega ) $ with 
$$
\theta = n_{\gamma}+  \tilde \eta +\frac{ \lambda -n_{\gamma }}{p}\, .
$$

If $\theta >n_{\gamma }$, that is  $ p       \tilde \eta   >n_{\gamma}-\lambda    $,  by  the original result of \cite[Theorem~3.I]{da}  we deduce that 
$u\in C^{0,\alpha }(\bar \Omega , \delta_{\gamma} )$ with 
$$
\alpha =    \tilde \eta  +\frac{ \lambda  -n_{\gamma }}{p}.
$$
Note that for applying  \cite[Theorem~3.I]{da}  we need only condition \eqref{lemmacusp1}. 

In conclusion,  the following variant of Theorem~\ref{daprato} holds. 

\begin{theorem} \label{dapratovariant} Let $\Omega$ be a bounded domain in ${\mathbb R}^n$ such that condition \eqref{lemmacusp1} holds, and let  $p\in [1, \infty [$.  Assume  that  the $p$-Poincar\'{e} inequality \eqref{poin} holds.   
Let 
 $\gamma \in ]0, 1]$ and  $\lambda  >0$  be such that 
 $$
 p      \tilde \eta   > n_{\gamma} -\lambda \, .
 $$
 Then  there exists $c>0$ such that for all $f\in W^{1,\lambda}_{p,\gamma}(\Omega ) $ and for all $x,y\in \Omega $ we have
$$
 |f(x)-f(y)|\le c\| f\|_{W^{1,\lambda}_{p,\gamma}(\Omega ) } (|\bar x-\bar y|^{\gamma}+|x_n-y_n|)^{     \tilde \eta  +\frac{\lambda -n_{\gamma } }{p} }\, .
 $$
\end{theorem}

 We observe that inequality \eqref{poincare1} implies the validity of inequality \eqref{poin}  with   $\tilde \eta = \frac{n_{\gamma}}{n}+n-n_{\gamma}  $,   which is the constant $\eta$ used in Theorem~\ref{daprato}.
We also note that assuming the validity of $p$-Poincar\'{e} inequalities of type \eqref{poin} is nowadays standard in Analysis on Metric Spaces. For instance, we refer to  the celebrated paper \cite{koskela} where general   $p$-Poincar\'{e} inequalities of the form
\begin{equation}\label{poinko}
\dashint_{  B } | f- f_{B} |d\mu \le c_p r\left(\dashint_{ \tau B} g^pd\mu \right)^{\frac{1}{p}}
\end{equation}
are considered. Here $g$ is the upper gradient of $f$, $B$ is an arbitrary   ball of radius $r$ in a metric space space $X$, $\tau B$ is the concentric  ball of radius $\tau r$ for a fixed $\tau \ge 1$ and $\mu$ is a suitable measure in $X$. Sufficient conditions ensuring the validity of \eqref{poinko} are known in the literature and are discussed e.g., in \cite[\S~10]{koskela}.  See also \cite{durand} for a  more recent work on this subject.  
We note that the study of inequalities of the type  \eqref{poin} in domains with cusps or domains of class $C^{0,\gamma}$ is very delicate and in general one does not expect their validity, in particular for outer cusps. Conditions for the validity of global  $(p,p)$-Poincar\'{e} inequalities (which means that the power $p$ appears also in the left-hand side of \eqref{poinko}) in domains with inner cusps and more generally John domains or $L^p$-averaging domains are given in \cite{staple}  where, besides an interesting counterexample,  a class of domains admitting moderately sharp outer `spires'  is also analyzed.

 \section{  Extension of Sobolev-Morrey spaces for $C^{0,\gamma}$ domains  }

    \subsection{The case of elementary domains of class $C^{0,\gamma}$}
 
Let $\Omega$ be an  elementary H\"older continuous   domain    in ${\mathbb R}^n$ with exponent $\gamma\in ]0,1]$ as in \eqref{eldom}, with $W={\mathbb R}^{n-1}$ and $a=-\infty$. 
Following \cite{burpaper1, b}, we set  $G={\mathbb R}^n\setminus \overline \Omega$ and 
$$G_k=\{x\in G: 2^{-k-1}<\rho_n(x)\le 2^{-k}\}$$
for all  $k\in{\mathbb Z}$,
where $\rho_n(x)=x_n-\varphi(\overline x)$ is the signed distance from $x\in {\mathbb R}^n$ to $\partial  G$ in the $x_n$ direction and we consider 
 a  partition of unity  associated with the covering 
$\{G_k \}_{k\in {\mathbb Z}}$ of $G$ satisfying a number of properties. Namely, it is proved in \cite{burpaper1} that for every $k\in {\mathbb Z}$ there exists 
 $\psi_k\in C^{\infty}({\mathbb R}^n)$ such that 
\begin{itemize}
\item[(i)]
$\displaystyle{ \sum_{k=-\infty}^\infty\psi_k=\begin{cases}1, \,\, {\rm if}\,\, x\in G,\vspace{2mm}\\ 
0, \,\, {\rm if }\,\, x\notin G;\end{cases}}$   \vspace{2mm}

\item[(ii)]  $  G=\displaystyle \cup_{k=-\infty}^\infty {\rm supp} \psi_k $ and  the covering $\{{\rm supp} \psi_k\}_{k\in {\mathbb Z}}$ has multiplicity equal to $2$;\vspace{2mm}

\item[(iii)]  $G_k\subset {\rm supp} \psi_k\subset G_{k-1}\cup G_k\cup G_{k+1}$, for all $k\in {\mathbb Z}$;\vspace{2mm}

\item[(iv)]  $|D^\alpha\psi_k(x)|\le c(\alpha) 2^{k\left(\frac{|\bar \alpha|}{\gamma }+\alpha_n \right)  }$, for all $x\in {\mathbb R}^n,k\in {\mathbb Z}, \alpha\in {\mathbb N}^n_0$. 
\end{itemize}

Note the appearance of  $\gamma$ in the exponent in item (iv) above. 

 Burenkov's Extension Operator   was defined in  \cite{burpaper1} as follows.  Let $l\in{\mathbb N}$ and $1\le p\le \infty$. For every $f\in W^{l,p}(\Omega)$, we set 
\begin{equation}
\label{burext}(Tf)(x)=\begin{cases}f(x), \,\, {\rm if}\,\,x \in \Omega,\vspace{2mm}\\
\displaystyle\sum_{k=-\infty}^\infty \psi_k(x)f_k(x),  \,\, {\rm if}\,\,x \in G,
\end{cases}
\end{equation}
where 
\begin{multline*}
f_k(x)=\int_{{\mathbb R}^n}f(\overline x -2^{-\frac{k}{\gamma}}\overline z,x_n-A2^{-k}z_n)\omega(z)dz =\\
=A^{-1}2^{\frac{k}{\gamma}(n-1)+k}\int_{{\mathbb R}^n}\omega(2^{\frac{k}{\gamma}}(\overline x-\overline y), A^{-1}2^{k}(x_n-y_n))f(y)dy\, ,
\end{multline*}  
  $A$ is a sufficiently large constant depending only on $n$ and $M$ in \eqref{lip1}  (in \cite{burpaper1} it is chosen for example $A=200(1+Mn)$) and   $\omega\in C^\infty_c({\mathbb R}^n)$ is a kernel of mollification 
defined by $$\omega (x)=\omega_1(x_1)\cdots \omega_n(x_n),\  \omega_i\in C^{\infty}_c(1/2,1),\  \int_{-\infty}^{+\infty}\omega (x_i)dx_i=1,\  \int_{-\infty }^{+\infty }\omega_i(x_i)x_i^kdx_i=0  $$ for all $i=1,\dots ,n$, $k=1,\dots ,l$.  

 Among other results (in particular, concerning anisotropic Sobolev spaces), it is proved in  \cite{burpaper1} that 
the operator $T$ is a linear continuous operator from $W^{l}_p(\Omega)$ to $W^{[\gamma l]}_p({\mathbb R}^n)$ where $[\gamma l]$ is the integer part of 
$\gamma l$. \\

The following theorem is a generalisation of the extension theorem pro\-ved in \cite{fala} in the case of Lispchitz domains, that is for  $\gamma =1$. 
Considering a number of technical issues appearing in the case $\gamma <1$, we assume for simplicity that the function $\phi$ defining the Morrey norm satisfies the condition 
 $\phi (r)=1$ for all $r>1$.

\begin{theorem}\label{sobolevmorrey}
Let $\Omega$ be an  elementary H\"older continuous   domain    in ${\mathbb R}^n$ with exponent $\gamma\in ]0,1]$,  with $W={\mathbb R}^{n-1}$ and $a=-\infty$. 
Let  $l\in {\mathbb N}$,  $p\in [1, \infty [$,   and   $\phi : ]0, \infty [ \to ]0,\infty [$ satisfy the condition $\phi (r)=1$ for all $r>1$. Then the operator $T$ maps $W^{l,\phi }_{p , 1  }(\Omega )$ continuously to   $W^{[\gamma l],\phi_{\gamma} }_{p , 1 }({\mathbb R}^n )$, where $\phi_{\gamma}$ is defined by
$\phi_{\gamma}(r)=\phi (r^{\gamma})$ for all $r\geq 0$.  In particular, $T$ maps the space  $W^{l,\lambda }_{p,1}(\Omega )$ to the space $W^{[\gamma l],\gamma \lambda }_{p,1}(\mathbb R^n )$, for any $ \lambda \geq 0$.
\end{theorem}

 The proof of Theorem~\ref{sobolevmorrey} can be carried out by adapting  the corresponding proof of \cite{fala} in a suitable way. Since the adaptation is quite technical and 
touches a number of delicate points, we indicate here the main steps starting from the first  but crucial lemmas  which we combine in the following statement. Here  $\tilde G_k=G_{k-1}\cup  G_k\cup  G_{k+1} =\{x\in G: 2^{-k-2}<\rho_n(x)\le 2^{-k+1}\}$ for all 
 $k\in {\mathbb Z}$ and ${\rm diam}\, C$ denotes the Euclidean diameter of a set $C$.

\begin{lemma}\label{controllo_geo} 
Assume that $B_{ 1  }(x,r)\cap G\ne\emptyset$
for some $x\in {\mathbb R}^n$ and $r>0$. Let $h\in {\mathbb Z}$ be the minimal integer such that $B_{ 1  }(x,r) \cap G_h\ne\emptyset$. Let $k\in {\mathbb Z}$  be such that $k\ge h+3$ and $B_{ 1  }(x,r) \cap  \tilde G_k  \ne\emptyset$. Then 
\begin{equation}\label{due}
|2^{-(h+3)}-2^{-k}|\le c (r+r^{\gamma})\,,
\end{equation}
where  $c$ depends only on $\gamma$ and   ${\rm Lip}_{\gamma}\varphi$.

Moreover, given $E>0$    there exists  $S>0$ depending only on $\gamma$, ${\rm Lip}_{\gamma}\varphi$, $ E$,   and a lower bound for $h$ such that      for every $\eta\in {\mathbb R}^n$, with $|\eta|<E$, 
\begin{equation}\label{inc0}
{\rm diam }\left( 
\bigcup_{k=h+3}^\infty \left( B_{ 1  }(x,r)  \cap \tilde G_k-   (2^{-\frac{k}{\gamma}} \bar  \eta,  2^{-k}   \eta_n )    \right) \right) \le S(r+r^{\gamma})\,.
\end{equation}
\end{lemma} 
\begin{proof}
By our assumptions we deduce that  $\{x\in B_{ 1  }(x,r):\, \rho_n(x)=2^{-h-2} \}  , \{x\in B_{ 1 }(x,r):\, \rho_n(x)=2^{-k+1} \}\ne \emptyset$ hence there exist  $y,w\in B_{ 1  }(x,r)$ with $y_n-\varphi (\bar y)=2^{-h-2}$ and $w_n-\varphi(\bar w)=2^{-k+1} $. Since $|y_n-x_n|, |\bar y-\bar w|<2r $, by the H\"{o}lder continuity of $\varphi$ we get
\begin{multline*}
|2^{-(h+3)}-2^{-k}|=\frac{1}{2}|2^{-h-2}-2^{-k+1}|=\frac{1}{2}|y_n-\varphi (\bar y)-w_n+\varphi (\bar w)|\\
\le \frac{1}{2}(|y_n-w_n|+{\rm Lip}_{\gamma}\varphi |\bar y-\bar w|^{\gamma} )\le \frac{1}{2}\left(   2r   +{\rm Lip}_{\gamma}\varphi (2r)^{\gamma}  \right)
\end{multline*}
and \eqref{due} follows.

We now prove \eqref{inc0}. Let $k\ge h+3$ be such that  $B_{ 1  }(x,r) \cap \tilde G_k\ne\emptyset$. Let $a\in B_{ 1  }(x,r) \cap \tilde G_{h+3}$ and $b\in B_{ 1  }(x,r) \cap \tilde G_{k}$. By \eqref{due}, for all $ \eta\in {\mathbb R}^n$, with $|\eta| <E$,  we have 
\begin{multline*}
|b_n- 2^{-k}\eta_n- (a_n-2^{-(h+3)}\eta_n)|\le |b_n-a_n|+ |2^{-k}-2^{-(h+3)}||\eta_n|\\
\le  2r+   cE(r+r^{\gamma})  \,
\end{multline*}
and 
 \begin{multline*}
|\bar b- 2^{-\frac{ k}{\gamma}  }\bar \eta- (\bar a-2^{- \frac{ h+3}{\gamma}}\bar \eta)|\le   |\bar b-\bar a|+ |2^{-\frac{k}{\gamma}}-2^{-\frac{h+3}{\gamma}}||\bar \eta|\\
  \le c\max\{1, 2^{-h(1-\gamma)/\gamma  }\} ( r +r^{\gamma})  \end{multline*}
 which proves \eqref{inc0}. 
 \end{proof}

Another crucial step in the proof is  \cite[Lemma~2.4, (ii)]{fala} which has to be modified as follows. 
As in \cite[Chap.~6]{b}, for every $k\in {\mathbb Z}$ we set 
\begin{equation*}\tilde \Omega_k=\{x\in \Omega: 2^{-k-2}<|\rho_n(x) |\le b2^{-k+1}\}\,, 
\end{equation*} where $b=10A$.

\begin{lemma}\label{mainlem} Assume that $B_{ 1 }(x,r)\cap G\ne\emptyset$
for some $x\in {\mathbb R}^n$ and $r>0$. 
Let $f\in W^{l,p}(\Omega)$ and ${\mathcal{U}}\subset {\mathbb R}^n$ be a  fixed measurable set with  $d:=\sup\{\rho_n(x):\ x\in B_1(x,r)\cap {\mathcal{U}} \}<\infty $.  
Then there exists  $c>0$ and $m\in {\mathbb N}$   depending only on $n$, $l$, $p$, $M$,  $\omega$, $ d$, and  
for every $\alpha \in {\mathbb N}^n_0$  with $|\alpha |\le l$ there exists a function $g_\alpha$ independent of $r, {\mathcal {U}}$, such that 
for every $z\in {\mathbb R}^n$ with $|z| \le c$ there exist $m $   balls  $B_{ 1}(x^{(i)}_{z} , r^{\gamma})$,  $i=1,\dots,m $,  such that  
\begin{multline}\label{mainlem2}
\| D^{\alpha}f_k-g_{\alpha}\|^p_{L^p(B_{1 }(x,r)\cap {\mathcal{U}}\cap \tilde G_k)}\\
\le c2^{pk( \frac{ |\bar \alpha |}{\gamma}+\alpha_n-l)} \int_{|z|\le c } \sum_{|\beta |=l}  \| D^{\beta }f\|^p_{L^{p}(\cup_{i=1}^{m}   B_{ 1}(x^{(i)}_z,  r^{\gamma} )\cap \tilde \Omega_k) } dz,  
\end{multline}
for all $k\in {\mathbb N}$. 
\end{lemma}

The proof of the previous lemma follows the lines of   \cite[Lemma~2.4, (ii)]{fala}.  We omit the lengthy  details but we explain  how this lemma is used and how the  modified exponent
$pk( \frac{ |\bar \alpha |}{\gamma}+\alpha_n-l)$  affects the final result.  Namely, in order to  prove Theorem~\ref{sobolevmorrey}, one has to estimate the derivatives  $D^{\alpha}Tf$  of the extension $Tf$ of a function $f$. By applying the  Leibnitz rule one ends up with estimating 
$D^{\alpha -\beta}\psi_k D^{\beta }f_k  $ for all $\beta \le \alpha$. The difficult part of the work concerns the case $\beta <\alpha $ and $k>0$. One observes that 
$\sum_{k\in {\mathbb Z}}D^{\alpha -\beta}\psi_k D^{\beta }f_k=\sum_{k\in {\mathbb Z}}D^{\alpha -\beta}\psi_k (D^{\beta }f_k -g_{\beta})$  for $\beta <\alpha$   since  $g_{\beta}$ does not depend on $k$. Thus, one has to estimate $D^{\beta }f_k -g_{\beta}$.
By combining the previous lemma with property (iv) of the partition of unity, we have
\begin{multline}\label{mainlem4}
    \| D^{\alpha -\beta}\psi_k( D^{\beta }f_k-g_{\beta}  )  \|^p_{L^p(B_1(x,r)\cap {\mathcal{U}}\cap \tilde G_k )}\\ 
\le c  2^{  p k(\frac{|\bar \alpha -\bar \beta|}{\gamma} +\alpha_n-\beta_n  )} \| D^{\beta }f_k-g_{\beta}    \|^p_{L^p(B_1(x,r)\cap {\mathcal{U} }\cap \tilde G_k )}\\
\le  c  2^{  p k(\frac{|\bar \alpha -\bar \beta|}{\gamma} +\alpha_n-\beta_n  )}
2^{pk( \frac{ |\bar \beta |}{\gamma}+\beta_n-l)}
  \int_{|z|\le c}  \sum_{|\beta|=l}\|D^{\beta} f\|^p_{L^{p}(\cup_{i=1}^{m } B_{ 1}(x^{(i)}_z,r^{\gamma})\cap \tilde \Omega_k ) }  dz\, .
\end{multline}
We note that the exponent of the power of $2$ in the right-hand side of \eqref{mainlem4} equals
$$  p k\biggl(\frac{|\bar \alpha -\bar \beta|}{\gamma} +\alpha_n-\beta_n  \biggr)
+pk\left( \frac{ |\bar \beta  |}{\gamma}+\beta _n-l\right)  = pk \left( \frac{|\bar \alpha |}{\gamma} +\alpha_n -l  \right) $$ 
hence one can control the right-hand side of \eqref{mainlem4}, provided that exponent is non-positive, that is 
\begin{equation}\label{indexbound}
|\bar \alpha | +\gamma \alpha_n\le \gamma l\, .
\end{equation}

Inequality \eqref{indexbound} explains why one gets $[\gamma l]$ as index of smoothness in the target Sobolev space 
$W^{[\gamma l],\phi }_{\lambda , \gamma }({\mathbb R}^n )$ in Theorem~\ref{sobolevmorrey}.  

 Moreover, in estimate \eqref{mainlem2} we have the quantity 
$$
 \| D^{\beta }f\|^p_{L^{p}(\cup_{i=1}^{m}   B_{ 1}(x^{(i)}_z, r^{\gamma} )\cap \tilde \Omega_k) } 
 $$
 and, since the balls have radius $r^{\gamma}$, one eventually controls that quantity via 
 $$
 \phi (r^{\gamma} ) \|  D^{\beta }f \|^p_{L^{\phi}_{p,1}(\Omega )  } 
 $$
which explains the  appearance of the new weight $\phi_{\lambda}$ in Theorem~\ref{sobolevmorrey}.   
For further details, we refer to the proof of \cite[Theorem 2.5]{fala}.\\

\vspace{12pt}

\subsection{ The case of general domains of class $C^{0,\gamma}$}

We recall the definition of open sets with $C^{0,\gamma}$  boundary. Here and in the sequel, given a set $C$ in ${\mathbb R}^n$ and $d>0$ we denote by $C_d$ the set
$\{x\in C:{\rm dist} (x,\partial C)>d\}$.

\begin{definition} Let $\gamma\in ]0,1]$,  $d >0$, $M\geq 0$, $s\in {\mathbb N} \cup \{\infty \}$. Let $\{V_{j}\}_{j=1}^s$ be a family of cuboids, i.e. for every $j=\overline{1,s}$ there exists 
an isometry $\lambda_j$ in ${\mathbb R}^n$ such that 
$$
\lambda_j( V_j  )= \Pi_{i=1}^n]a_{i,j}, b_{i,j}[
$$ 
where $0<a_{i,j}<a_{i,j}+d<b_{i,j}$. Assume that $D:=\sup_{j=\overline{1,s}}{\rm diam }V_j< \infty $, $(V_j)_{d}\ne \emptyset $ for all $j=\overline{1,s}$,   and that the multiplicity of the covering $\{V_{j}\}_{j=1}^s$ is finite. 
We then say that ${\mathcal{A} }=(s,d, \{V_{j}\}_{j=1}^s, \{\lambda_{j}\}_{j=1}^s )$ is an atlas.  

Let $M\geq 0$. We say that an open set $\Omega$ in ${\mathbb R}^n$ is of class $C^{0,\gamma}_M({\mathcal{A}})$  if the following conditions are satisfied:\\
 
(i) For every  $j=\overline{1,s}$, we have $\Omega \cap (V_j)_d\ne \emptyset$. \\

(ii) $\Omega \subset \cup_{j=1}^{s}(V_j)_d$.\\

(iii) For every  $j=\overline{1,s}$, the set ${\mathcal{H}}_j:=\lambda_j(\Omega \cap V_j)$  satisfies the following condition: either ${\mathcal{H}}_j= \Pi_{i=1}^n]a_{i,j}, b_{i,j}[ $ (in which case $V_j\subset \Omega $),  or ${\mathcal{H}}_j$ is a bounded  elementary H\"{o}lder continuous domain 
  of the form
\begin{equation*}\label{ele1bis}
{\mathcal{H}}_j=\left\{x\in {\mathbb R}^n:\ \bar x\in W_j,\ a_{n,j}<x_n<\varphi_j (\bar x) \right\}
\end{equation*}
where  $\varphi_j $ is a real-valued H\"{o}lder continuous function with exponent $\gamma$,  defined on $W_j= \Pi_{i=1}^{n-1}]a_{i,j}, b_{i,j}[$ such that
$$
a_{n,j}+d<\varphi_j\ \ {\rm and}\ \ {\rm Lip }_{\gamma}\varphi_j \le M 
$$ 
(in which case $V_j\cap \partial \Omega \ne \emptyset$). 

Finally, we say that an open set $\Omega$ in ${\mathbb R}^n$ is of class $C^{0,\gamma}$ if it is of class  $C^{0,\gamma}_M({\mathcal{A}})$ for some $M$ and ${\mathcal A}$.
\end{definition}

The definition of Burenkov's Extension Operator for a general domain of class $C^{0,\gamma}$ is given by pasting together the extension operators
defined on each chart of the atlas as follows. 
Following \cite[p.265]{b}, given an open set $\Omega $ of class $C^{0,\gamma}_M({\mathcal{A}})$, we consider a family of   functions $\{\psi\}_{j=1}^s$  such that $\psi_j\in C^{\infty }_c({\mathbb R}^n)$, $ {\rm supp} \psi_j\subset  (V_j)_{d} $, $0\le \psi_j\le 1$, $\sum_{j=1}^s\psi_j^2(x)=1$ for all $x\in \Omega$ and such that $\| D^{\alpha }\psi_j\|_{L^{\infty }({\mathbb R}^n)}$ $\le M$ for all $j=\overline{1,s}$ and $\alpha\in {\mathbb N}_0^n$ with $|\alpha |\le l$, where $M$ depends only on $n,l,d$.   

Burenkov's Extension Operator $T$ is defined from $W^{l}_p(\Omega )$ to $W^{[\gamma l]}_p({\mathbb R}^n)$ by 
\begin{equation}
\label{burextgen}
Tf= \sum_{j=1}^s\psi_j  T_j(f\psi_j),
\end{equation}
for all $f\in W^{l,p}(\Omega )$, where $T_j$ are the extension operators defined on each  domain $\Omega \cap V_{j}$. See \cite{fala} for details.

Then, we have he following. Recall that $\phi_{\gamma}$ is defined by $\phi_{\gamma}(r)=\phi (r^{\gamma})$ for all $r\geq 0$.

\begin{theorem}\label{sobolevmorreybis} Let $\Omega$ be an open set    in ${\mathbb R}^n$ of class $C^{0,\gamma}$ with  $\gamma\in ]0,1]$. Let  $l\in {\mathbb N}$,  $p\in [1, \infty [$,   and   $\phi : ]0, \infty [ \to ]0,\infty [$ satisfying the condition $\phi (r)=1$ for all $r>1$. Then the operator $T$ maps $W^{l,\phi }_{p , 1  }(\Omega )$ continuously to $W^{[\gamma l],\phi_{\gamma} }_{p , 1 }({\mathbb R}^n )$. In particular, $T$ maps the space  $W^{l,\lambda }_{p,1}(\Omega )$ to the space $W^{[\gamma l],\gamma \lambda }_{p,1}(\mathbb R^n )$, for any $\lambda \geq 0$. 

\end{theorem}

The proof of Theorem~\ref{sobolevmorreybis} can be carried out by pasting together local extensions operators provided by  Theorem~\ref{sobolevmorrey} in each cuboid
of the covering of $\Omega$. This argument is described in detail in the proof of \cite[Theorem~3.3]{fala}. Finally, we can deduce the following

\begin{corollary}\label{maincor}  Let $\Omega$ be an open set    in ${\mathbb R}^n$ of class $C^{0,\gamma}$ with  $\gamma\in ]0,1]$. Let  $l\in {\mathbb N}$,  $p\in [1, \infty [$, and $\lambda>0$.  If 
 $$
 p [\gamma l]> n -\gamma \lambda \, 
 $$
 and $ [\gamma l]+\frac{\gamma \lambda -n }{p}   <1$  then there exists $c>0$ such that   for all $f\in W^{l,\lambda}_{p,1}(\Omega ) $ and for all $x,y\in \Omega $ we have
 \begin{equation}\label{besov1tris}
 |f(x)-f(y)|\le c \|f\|_{W^{l,\lambda}_{p,1}(\Omega ) }  |x-y|^{  [\gamma l]+\frac{\gamma \lambda  -n}{p}   }\, .
 \end{equation}
\end{corollary}

The proof of the previous corollary follows immediately by Theorem~\ref{sobolevmorreybis}  and estimate \eqref{besov1} applied with $\gamma =1$ and $l$ replaced 
by $ [\gamma l]$. Indeed, by Theorem~\ref{sobolevmorreybis}, any functions $f\in W^{l,\lambda}_{p,1}(\Omega ) $ is extended to the whole of $\mathbb{R}^n$ as a function 
of  $W^{[\gamma l],\gamma \lambda }_{p , 1 }({\mathbb R}^n )$ to which the classical Sobolev-Morrey Theorem applies.

\section*{ Acknowledgments}  The authors are very thankful to Prof. Victor I. Burenkov for useful discussions and  for suggesting the study  of the action of his extension operator on Sobolev-Morrey spaces defined on domains with H\"{o}lder continuous boundaries.  The authors are members of the Gruppo Nazionale per l'Analisi Matematica, la Probabilit\`a e le loro Applicazioni (GNAMPA) of the Istituto Nazionale di Alta Matematica (INdAM).


\begin{thebibliography}{20}
\thispagestyle{headings}
\footnotesize

 
 
 \bibitem{ba} G.C.~Barozzi, {\it Su una generalizzazione degli spazi $L^{(q,\lambda )}$  di Morrey.}  (in Italian) Ann. Scuola Norm. Sup. Pisa (3) 19 (1965), 609--626. 

\bibitem{be1} O.V.~Besov, V.P.~Il'in, S.M.~Nikolskii,  {\it Integral representations of functions and imbedding theorems}. Vol. I. Translated from  Russian. Scripta Series in Mathematics. Edited by M.H. Taibleson. V.H. Winston \& Sons, Washington, D.C.; Halsted Press [John Wiley \& Sons], New York-Toronto, Ont.-London, 1978.
 
\bibitem{be2}  O.V.~Besov, V.P.~Il'in, S.M.~Nikolskii,  {\it Integral representations of functions and imbedding theorems}. Vol. II. Translated from  Russian. Scripta Series in Mathematics. Edited by M.H. Taibleson. V.H. Winston \& Sons, Washington, D.C.; Halsted Press [John Wiley \& Sons], New York-Toronto, Ont.-London, 1979. 

\bibitem{burpaper2} V.I.~Burenkov,  {\it The continuation of functions with preservation and with deterioration of their differential properties}. (in Russian) Dokl. Akad. Nauk SSSR  224 (1975), no. 2, 269--272.  English transl. in Soviet Math. Dokl. 16 (1975). 

\bibitem{burpaper1} V.I.~Burenkov,  {\it A way of continuing differentiable functions}. (in Russian) Studies in the theory of differentiable functions of several variables and its applications, VI. Trudy Mat. Inst. Steklov. 140 (1976), 27--67, 286--287.  English transl. in Proc. Steklov Inst. Math., American Mathematical Society, Providence, Rhode Island, 140 (1979, issue 1).  

\bibitem{b} V.I.~Burenkov,  {\it Sobolev Spaces on Domains.} Teubner-Texte Zur Mathematik, 137 (1998) Springer.



\bibitem{CZ} A.P.~Calderon, A.~Zygmund, {\it  On singular integrals.} American Journal of Mathematics, The Johns Hopkins University Press, 78 (2) (1956), 289--309.


\bibitem{campa3}  S.~Campanato, {\it Il teorema di immersione di Sobolev per una classe di aperti non dotati della propriet\`{a} di cono.} (in Italian) Ricerche Mat. 11 (1962), 103--122. 

\bibitem{campa1}  S.~Campanato, {\it Propriet\`{a} di h\"{o}lderianit\`{a} di alcune classi di funzioni.}  (in Italian) Ann. Scuola Norm. Sup. Pisa (3) 17 (1963), 175--188.

\bibitem{campa2}  S.~Campanato, {\it Propriet\`{a} di inclusione per spazi di Morrey.} (in Italian)  Ricerche Mat. 12 (1963), 67--86. 


\bibitem{CS} S.~Campanato,  G.~Stampacchia, {\it Sulle maggiorazioni in $L^p$ nella teoria delle equazioni ellittiche.} (in Italian) Bollettino dell'Unione Matematica Italiana, Serie 3, Vol. 20 (1965), n.3, p. 393-399. Bologna, Zanichelli, 1965.


\bibitem{C} S.~Campanato, {\it Sistemi ellittici in forma divergenza: regolarit\`{a} all'interno.} (in Italian) Pisa, Scuola Normale Superiore editors, 1980.


\bibitem{CTV} P.~Cannarsa, B.~Terreni, V.~Vespri, {\it  Analytic semigroups generated by nonvariational elliptic systems of second order under Dirichlet boundary conditions.} J. Math. Anal. Appl. 112 (1985), no. 1, 56--103.

\bibitem{Co}  R.R.~Coifman, {\it A real variable characterization of $H^p$.} Studia Math,  51 (1974), 269--274.

\bibitem{Co1}  R.R.~Coifman {\it Characterization of Fourier transforms of Hardy spaces.} Proc Nat Acad Sci U S A, 1974, 71 (1974),  4133--4134.

\bibitem{CoW}  R.R.~Coifman, G.~Weiss, {\it  Analyse harmonique non-commutative sur certains espaces homog\`{e}nes.}  In: Lecture Notes in Math 242. Berlin-New York: Springer-Verlag, 1971.

\bibitem{CoW1}  R.R.~Coifman, G.~Weiss, {\it  Extensions of Hardy spaces and their use in analysis.} Bull Amer Math Soc, 83 (1977),  569--645.

\bibitem{da} G.~Da Prato, {\it  Spazi $ {\mathcal{L}}^{p, \theta}(\Omega, \delta )$ e loro propriet\`{a}.}  (in Italian) Ann. Mat. Pura Appl. (4) 69  (1965),  383--392. 

\bibitem{DG} E.~De Giorgi, {\it  Sulla differenziabilit\`{a} e l'analiticit\`{a} delle estremali degli integrali multipli regolari.}  (in Italian) Mem. Accad. Sci. Torino Cl. Sci. Fis. Mat. Nat., (3) 3 (1957), 25--43.

\bibitem{durand} E.~Durand-Cartagena, J.A.~Jaramillo, N.~Shanmugalingam, {\it  First order Poincar\'{e} inequalities in metric measure spaces.} Ann. Acad. Sci. Fenn. Math. 38 (2013), no. 1, 287--308.

\bibitem{fala} M.S.~Fanciullo, P.D.~Lamberti, {\it On Burenkov's extension operator preserving Sobolev-Morrey spaces on Lipschitz domains.}  Math. Nachr. 290 (2017), no. 1, 37--49. 

\bibitem{FS} C.~Fefferman, E.M.~Stein, {\it  $H^p$  spaces of several variables.} Acta Math., 129 (1972), 137--193.

\bibitem{G}  M.~Giaquinta, {\it  Multiple integrals in the calculus of variations and nonlinear elliptic systems.} in Annals of Mathematics Studies, 105. Princeton University Press, Princeton, NJ, 1983.


\bibitem{gitr} D.~Gilbarg, N.S.~Trudinger, {\it  Elliptic partial differential equations of second order.} Reprint of the 1998 edition. Classics in Mathematics. Springer-Verlag, Berlin, 2001.

\bibitem{greco} D.~Greco, {\it Criteri di compattezza per insiemi di funzioni in n variabili indipendenti.}  Ricerche Mat. 1, (1952), 124--144. 

\bibitem{koskela} P.~Haj\l asz, P.~Koskela, {\it   Sobolev met Poincar\'{e}.}  Mem. Amer. Math. Soc. 145 (2000), no. 688.


\bibitem{koskzhou}  P.~Koskela,  Y.\,R-Y.~Zhang, Y.~Zhou, {\it  Morrey-Sobolev extension domains.}
  J. Geom. Anal.  27 (2017) no. 2, 1413--1434.

\bibitem{kufner}  A.~Kufner, O.~John, S.~Fu\v{c}ik, {\it Function spaces.}  Monographs and Textbooks on Mechanics of Solids and Fluids; Mechanics: Analysis. Noordhoff International Publishing, Leyden; Academia, Prague, 1977.

\bibitem{lapi} P.D.~Lamberti, Y.~Pinchover, {\it  $L^p$ Hardy inequality on $C^{1,\gamma}$ domains.} To appear in   Ann. Scuola Norm. Sup. Pisa (5), 19 (2019), no. 3, 1135--1159.

\bibitem{lavio} P.D.~Lamberti, I.Y.~Violo, {\it  On Stein's extension operator preserving  Sobolev-Morrey spaces.} Math. Nachr.  292 (2019), no. 8, 1701--1715.

\bibitem{L} R.H.~Latter, {\it  A characterization of $H^p(\mathbb{R}^n)$ in terms of atoms.} Studia Math, 62 (1978),  93--101.

\bibitem{NS} E.~Nakai, Y.~Sawano, {\it  Orlicz-Hardy spaces and their duals.} Sci China Math,  57 (2014),  903--962.

\bibitem{N} J.~Nash, {\it  Continuity of Solutions of Parabolic and Elliptic Equations  American Journal of Mathematics.} Vol. 80, No. 4. ( 1958), pp. 931--954.

\bibitem{mazpob}  V.G.~Maz'ya, S.V.~Poborchi, {\it  Differentiable functions on bad domains.} World Scientific Publishing Co., Inc., River Edge, NJ, 1997.

\bibitem{nirenberg} L.~Nirenberg, {\it Estimates and existence of solutions of elliptic equations.} Comm. Pure Appl. Math. 9 (1956), 509--529. 

\bibitem{staple}   S.G.~Staples, {\it $L^p$-averaging domains and the Poincar\'{e} inequality.} Ann. Acad. Sci. Fenn. Ser. A I Math. 14 (1989), no. 1, 103--127. 

\bibitem{S}  E.M.~Stein, {\it  Singular Integrals and Differentiability Properties of Functions.} Princeton, NJ: Princeton University Press, 1970.

\bibitem{S1}  E.M.~Stein, {\it Harmonic Analysis: Real-variable Methods, Orthogonality.} and Oscillatory Integrals. Princeton, NJ: Princeton University Press, 1993.
 
\bibitem{SW}  E.M.~Stein, G.~Weiss, {\it On the theory of harmonic functions of several variables.} I. The theory of $ H^p$ -spaces. Acta Math, , 103 (1960), 25--62.
 
\bibitem{TW}  M.H.~Taibleson, G.~Weiss, {\it  The molecular characterization of certain Hardy spaces.} In: Representation theorems for Hardy spaces. Ast\'{e}risque,  77 (1980)  67--149.

\bibitem{vitolo} A.~Vitolo, {\it  Functions with derivatives in spaces of Morrey type.}  Rend. Accad. Naz. Sci. XL Mem. Mat. Appl. {\bf 5} (1997),  21, 1--24.
  
\bibitem{W} T.~Walsh, {\it  The dual of $H^p (R^{ n+1}_ + )$ for $p < 1$.}  Canad. J. Math, 25 (1973),  567--577.

 

\end{thebibliography}
 \end{document}